\documentclass[a4paper]{amsart}

\usepackage{amsrefs}

\usepackage{etex}
\usepackage{stmaryrd}
\usepackage{hyperref}

\usepackage{txfonts,amsmath,amstext,amsthm,amscd,amsopn,verbatim,amssymb,amsfonts,mathtools,enumitem}
\usepackage{fullpage}

\usepackage{todonotes}

\usepackage{adjustbox}

\usepackage[bbgreekl]{mathbbol}

\usepackage{tikz}

\usetikzlibrary{matrix}
\usetikzlibrary{patterns, shapes}
\usetikzlibrary{arrows}
\usetikzlibrary{calc,3d}
\usetikzlibrary{decorations,decorations.pathmorphing, decorations.pathreplacing}
\usetikzlibrary{through}
\tikzset{ext/.style={circle, draw,inner sep=1pt},int/.style={circle,draw,fill,inner sep=1pt},nil/.style={inner sep=1pt}}
\tikzset{exte/.style={circle, draw,inner sep=3pt},inte/.style={circle,draw,fill,inner sep=3pt}}
\tikzset{diagram/.style={matrix of math nodes, row sep=3em, column sep=2.5em, text height=1.5ex, text depth=0.25ex}}
\tikzset{diagram2/.style={matrix of math nodes, row sep=0.5em, column sep=0.5em, text height=1.5ex, text depth=0.25ex}}
\tikzset{every picture/.style={baseline=-.65ex}}
\tikzstyle{every loop}=[draw]
\tikzstyle{rloop}=[ out=10, in=-10, loop, distance=3em] 
\tikzstyle{aloop}=[ out=100, in=80, loop, distance=3em] 
\usepackage{tikz-cd}

\theoremstyle{plain}

\newtheorem{thm}{Theorem}[section]

\newtheorem{prop}[thm]{Proposition}

\newtheorem{lemma}[thm]{Lemma}

\theoremstyle{definition}

\hypersetup{pdfauthor={Thomas Willwacher}}

\newcommand{\Q}{{\mathbb{Q}}}

\newcommand{\cP}{\mathsf{P}}

 \newcommand{\Feyn}{\mathrm{Feyn}}
 \newcommand{\AFeyn}{\mathrm{AFeyn}}

\newcommand{\BV}{\mathsf{BV}}
\newcommand{\Grav}{\mathsf{Grav}}
\newcommand{\tGrav}{\widetilde{\Grav}}

\renewcommand{\Bar}{{\mathtt{B}}}

\newcommand{\kk}{\mathfrak{K}}

\newcommand{\Com}{\mathsf{Com}}

\newcommand{\bpm}{\begin{pmatrix}}
\newcommand{\epm}{\end{pmatrix}}

\newcommand{\gr}{\mathrm{gr}}

\newcommand{\beq}[1]{\begin{equation}\label{#1}}
\newcommand{\eeq}{\end{equation}}

\newcommand{\ar}{\mathrm{ar}}

\newcommand{\M}{\mathcal{M}}
\newcommand{\MM}{\overline{\M}}

\DeclareMathAlphabet{\mathsfit}{OT1}{cmss}{m}{sl}

\DeclareMathOperator{\COp}{\mathsfit{C}}

\DeclareMathOperator{\POp}{{\mathsfit{P}}}

\author{Thomas Willwacher}
\address{Department of Mathematics \\ ETH Zurich \\
R\"amistrasse 101 \\
8092 Zurich, Switzerland}
\email{thomas.willwacher@math.ethz.ch}
\thanks{The author has been partially supported by the NCCR Swissmap, funded by the Swiss National Science Foundation}

\begin{document}
\title{Cyclic model for the dg dual of the BV operad}
\begin{abstract}
We describe a new small cyclic operad model $Q\BV$ for the dg dual operad $\Bar^c\BV^*$ of the Batalin-Vilkovisky operad.
As an application, we show that the Feynman transform of $\BV$ is quasi-isomorphic to the amputated Feynman transform of $\Bar^c\BV^*$, in complementary weights and degrees.
\end{abstract}
\maketitle

\section{Introduction}
The Batalin-Vilkovisky operad $\BV$ is the homology operad of the framed little 2-disks operad.
It is well known that the $\BV$ operad is a cyclic operad \cite{BudneyCyclic}.
This paper is concerned with the dg dual 1-shifted cyclic operad $\Bar^c\BV^*$, defined as the cyclic operadic bar construction of the dual operad $\BV^*$.
This operad $\Bar^c\BV^*$ has also been studied in the literature. 
There are fairly simple explicit models for $\Bar^c\BV^*$ as a non-cyclic operad, for example the Koszul dual operad $\BV^!$ \cite{GCTV}.
The minimal model $H(\Bar^c\BV^*)$ of $\Bar^c\BV^*$ has also been studied in the literature \cite{DCV}. However, the homotopy operad structure on $H(\Bar^c\BV^*)$ is not known very explicitly.

Hence we lack "small" explicit models for $\Bar^c\BV^*$ that are also cyclic operads.
The purpose of this paper is to fill this gap.
More precisely, define the symmetric sequence of dg vector spaces 
\[
Q\BV((r)) := 
\begin{cases}
    (\BV((r))[6-2r] \otimes \Q[v_1,\dots,v_r], d) & \text{if $r\geq 3$} \\
    \Q[u][1] & \text{if $r=2$}  
\end{cases},
\]
where $v_j$ are formal variables of cohomological degree $-2$, $u$ is a formal variable of degree $+2$, and the differential acts such that 
\[
d (x\otimes p(v_1,\dots, v_r))
=
\sum_{j=1}^r
(-1)^{|x|} (x\otimes_j \Delta) \otimes \frac{\partial}{\partial v_j} p(v_1,\dots, v_r).
\]
The expert reader may recognize that for $r\geq 3$
\[
    Q\BV((r)) \simeq \BV((r)) \sslash H_\bullet(S_1)^{\otimes r}    
\]
is just a model for the homotopy quotient of $\BV((r))$ by the circle actions. This is why we use the letter "$Q$".

Our first main result is then:
\begin{thm}\label{thm:main1}
There is a 1-shifted cyclic $\infty$-operad structure on $Q\BV$ with only $\leq 3$-ary composition operations, constructed in section \ref{sec:Qdual} below, such that there is a quasi-isomorphism
\[
    Q\BV \simeq \Bar^c\BV^*   
\]
of 1-shifted cyclic $\infty$-operads.
\end{thm}

We hence advertise $Q\BV$ as our small and completely explicit model of $\Bar^c\BV^*$. 
As an application of our model, we can show the following result that has been conjectured in \cite{BrueckBorinskyWillwacher}.

\begin{thm}\label{thm:main2}
The cohomology of the Feynman transform of $\BV$ and that of the amputated Feynman transform of $\Bar^c\BV^*$ are isomorphic in complementary degree and weights, 
\[
    \gr_{W} H^{-k}(\Feyn(\BV)((g,n)))
    \simeq
    \gr_{6g-6+2n-W} H^{6g-6+2n-k}(\AFeyn_{\kk}(\Bar^c\BV^*)((g,n)) ) 
    ,
\]
for all $W$ and all $(g,n)$ such that $2g+n\geq 3$.
\end{thm}

We note that the case $W=0$ of this statement is the main result of \cite{HainautPetersen}, while the case $W=2$ is the main result of \cite{BrueckBorinskyWillwacher}, which we generalize here to all weights. The right-hand side of the formula of Theorem \ref{thm:main2} computes the cohomology of the handlebody group \cite{Giansiracusa}, and the left-hand side is (somewhat) computable for low $W$ \cite{BrueckBorinskyWillwacher}. 

The rest of the paper is organized as follows.
In section \ref{sec:notation} we briefly recall necessary prerequisites and definitions.
Section \ref{sec:Qdual} contains the definition of the homtopy operad structure on $Q\BV$, and section \ref{sec:proof1} the proof of Theorem \ref{thm:main1}.
Curiously, the main ingredient of the proof are the topological recursion relations for genus zero intersection numbers.
Finally, the application to the Feynman transforms and the proof of Theorem \ref{thm:main2} can be found in section \ref{sec:feyn}.

\subsection*{Acknowledgements}
The author is grateful for discussions with Michael Borinsky, Benjamin Br\"uck, Sergei Merkulov, Dan Petersen and Marko \v Zivkovi\'c.

\section{Notation and prerequisites}\label{sec:notation}
\subsection{Differential graded vector spaces}
We work over the ground field $\Q$, that is, all vector spaces are understood over $\Q$ and (co)homology is taken with rational coefficients. 
We generally use cohomological conventions: All differentials on differential graded vector spaces will have degree $+1$.
We may convert from homological conventions by reversing the grading. For a graded vector space $V$ we write equivalently
\[
    V^{k} = V_{-k}     
\]
for the subspace of cohomological degree $k$ or homological degree $-k$.
The $j$-fold downwards degree shift is denoted by $V[j]$. For example if $V$ is concentrated in degree 0, then $V[j]$ is concentrated in degree $-j$.

\subsection{Multi-indices}
\label{sec:multiind}
For $\underline i=(i_0,i_1,\dots, i_n)\in \mathbb Z^n_{\geq 0}$ a multi-index we use the following notation.
\begin{align*}
|\underline i| &= i_0+\cdots+i_n
&
\underline i! = i_0!i_1! \cdots i_n!.
\end{align*}
We furthermore define a second multi-index 
\[
\lambda(\underline i)
=
(\lambda_1,\dots , \lambda_N)
=
(\underbrace{0,\dots,0}_{i_0\times},
\underbrace{1,\dots,1}_{i_1\times},
\dots,
\underbrace{n,\dots,n}_{i_n\times}),
\]
with $N=|\underline i|$.
For $\lambda=(\lambda_1,\dots, \lambda_N)$ a multi-index and $x_1,\dots,x_N$ variables we set 
\[
x^\lambda = x_1^{\lambda_1} \cdots x_N^{\lambda_N}.
\]

\subsection{Topological recursion relations for intersection numbers}
\label{sec:psi ints}
Consider the integrals of products of $\Psi$-classes over the Deligne-Mumford compactification of the moduli space of curves $\MM_{0,n}$,
\[
\int_{\MM_{0,n}} \Psi_1^{\lambda_1}\cdots \Psi_n^{\lambda_n} 
= 
\frac{(n-3)!}{\prod_\alpha \lambda_\alpha!} \delta_{|\lambda|,(n-3)}.    
\]
Here $\Psi_j$ denotes the $\Psi$-class at the $j$-th marking, and for the equality see \cite[Lemma 1.5.1]{Kock}.
Following standard conventions we denote, for $\underline i$ a multi-index with $|\underline i|=n$, 
\[
    \langle \tau^{\underline i}  \rangle_0
    =
\langle \tau_0^{i_0}\tau_1^{i_1}\cdots  \rangle_0
:=
\int_{\MM_{0,n}}
\Psi^{\lambda(\underline i)}
=
\int_{\MM_{0,n}}
\prod_{\alpha=1}^n\Psi_\alpha^{\lambda_\alpha}
=
\frac{(n-3)!}{
    \prod_{\alpha=0}^n \lambda_\alpha!
}
\delta_{|\lambda(i)|,(n-3)}.
\]
The genus zero topological recursion formula (see \cite[equation (2.69)]{Witten} or \cite[section 3.5]{Kock}) is the following identity for the expressions
$\langle \tau^{\underline i}  \rangle_0$:
\begin{equation}
    \label{equ:top recur 1}
\frac1{\underline i!}
\langle \tau_{a+1}\tau_b\tau_c\tau^{\underline i}  \rangle_0
=
\sum_{\underline i' +\underline i'' =\underline i}
\frac{1}{\underline i'!\underline i''!}
\langle \tau_{a}\tau_0\tau^{\underline i'}  \rangle_0
\langle \tau_b\tau_c\tau_0\tau^{\underline i''}  \rangle_0.
\end{equation}
We shall use a slightly simplified and more symmetric variant.
\begin{lemma}\label{lem:top recur}
    For $j=(j_0,j_1,\dots)$ any multi-index such that $j_0\geq 1$ and any $a,b\geq 0$ we have that 
    \begin{equation}
        \label{equ:top recur 2}
    \frac1{\underline j!}
    \left(
    \langle \tau_{a+1}\tau_b\tau^{\underline j}  \rangle_0
    +
    \langle \tau_{a}\tau_{b+1}\tau^{\underline j}  \rangle_0
    \right)
    =
    \sum_{\underline j' +\underline j'' =\underline j}
    \frac{1}{\underline j'!\underline j''!}
    \langle \tau_{a}\tau_0\tau^{\underline j'}  \rangle_0
    \langle \tau_b\tau_0\tau^{\underline j''}  \rangle_0.
    \end{equation}
\end{lemma}
\begin{proof}
We start from \eqref{equ:top recur 1} with $c=0$ and the multi-index $\underline i$ defined such that 
\[
\underline j= (i_0+1,i_1,i_2,\dots).
\]
Then, symmetrizing over $a$ and $b$ we find that 
\begin{align*}
    &\frac1{\underline j!}
    \left(
    \langle \tau_{a+1}\tau_b\tau^{\underline j}  \rangle_0
    +
    \langle \tau_{a}\tau_{b+1}\tau^{\underline j}  \rangle_0
    \right)
    = 
    \frac1{\underline j!}
    \left(
    \langle \tau_{a+1}\tau_b\tau_0\tau^{\underline i}  \rangle_0
    +
    \langle \tau_{a}\tau_{b+1}\tau_0\tau^{\underline i}  \rangle_0
    \right)
    \\&=
    \frac{1}{j_0}
    \sum_{\underline i' +\underline i'' =\underline i}
    \frac{1}{\underline i'!\underline i''!}
    \left(
    \langle \tau_{a}\tau_0\tau^{\underline i'}  \rangle_0
    \langle \tau_b\tau_0^2\tau^{\underline i''}  \rangle_0
    +
    \langle \tau_{a}\tau_0^2\tau^{\underline i'}  \rangle_0
    \langle \tau_b\tau_0\tau^{\underline i''}  \rangle_0
    \right)
    \\&=
    \frac{1}{j_0}
    \sum_{\underline j' +\underline j'' =\underline j}
    \langle \tau_{a}\tau_0\tau^{\underline j'}  \rangle_0
    \langle \tau_b\tau_0\tau^{\underline j''}  \rangle_0
    \frac{1}{\underline j'!\underline j''!}
    \left( j_0'' + j_0' \right)
    \\&=
    \sum_{\underline j' +\underline j'' =\underline j}
    \frac{1}{\underline j'!\underline j''!}
    \langle \tau_{a}\tau_0\tau^{\underline j'}  \rangle_0
    \langle \tau_b\tau_0\tau^{\underline j''}  \rangle_0.
\end{align*}
\end{proof}

\subsection{Cyclic operads}
We assume here that the reader has basic knowledge of cyclic and modular operads \cite{GK,GKcyclic}.
Notationally we follow the conventions of \cite[Section 2]{BrueckBorinskyWillwacher}, to which we refer for details.
A symmetric sequence $\POp$ is a collection $\POp((n))$ of right dg modules of the symmetric group $S_n$ for $n\geq 2$.
Informally, we think of $\POp((n))$ as a space of $n$-linear operations, and the $S_n$ action as permuting the inputs of these operations. 

Alternatively and equivalently, it is often convenient to define a symmetric sequence as a contravariant functor from the category of finite sets with bijections as morphisms to the category of dg vector spaces.
Then $\POp((n)):=\POp(\{1,\dots,n\})$ is the value of the functor on the set $\{1,\dots,n\}$, and conversely one may recover the functor on a set $A$ by setting
\[
\POp((A)) := \left(\bigoplus_{f: A\xrightarrow{\cong}\{1,\dots,n\}} \POp((n))  \right)_{S_rn},
\]
where the direct sum is over bijections from $A$ to $\{1,\dots,n\}$ and the symmetric group acts diagonally on the set of such bijections and on $\POp((n))$.
Intuitively, this corresponds to labeling the inputs of our operations by elements of $A$ rather than numbers.
We shall freely pass between both definitions of symmetric sequences and always use the most convenient.

A cyclic operad is a symmetric sequence $\POp$ together with binary composition operations
\[
  \circ_{a,b} : \cP((A))\otimes \POp((B))
  \to \POp((A\sqcup B \setminus \{a,b\})),
\]
and a unit element $1\in \POp((2))$.
The composition has to satisfy natural associativity and equivariance axioms, and the unit element behaves as the neutral element with respect to the composition.

A non-unital cyclic (peudo-)operad is the same data, but without the unit.
An augmented cyclic operad $\POp$ is a cyclic operad equipped with an augmentation morphism $\eta:\POp\to \mathbf{1}$, and the augmentation ideal 
\[
    \overline \POp 
    =
    \ker \eta    
\]
is a non-unital cyclic operad.
All our operads will be augmented without further mention. 

We also use the notion of 1-shifted cyclic operads, which are defined analogously to cyclic operads, except that the composition $\circ_{a,b}$ has cohomological degree $+1$.
Our notion of 1-shifted cyclic (or modular) operad corresponds to a $\kk$-cyclic (or modular) operad in \cite{GK} or other literature.

\subsection{Feynman transform and bar construction of cyclic operads}
The Feynman transform $\Feyn(\POp)$ is a version of the bar construction of a modular operad $\POp$.
In particular, $\Feyn(\POp)$ is a 1-shifted modular cooperad.
Again we do not fully recall the construction here for the sake of brevity, but refer to the literature \cite{GK}. Notationally we will again follow the conventions of \cite[Section 2]{BrueckBorinskyWillwacher}, except that we will use the symbols $\Bar$ and $\Bar^c$ for the bar and cobar construction of cyclic (co)operads, and not "$D$".

The quickest definition is that the Feynman transform of a modular operad $\POp$ is the cofree modular cooperad cogenerated by $\overline \POp$, with a differential encoding the modular operad structure on $\POp$.

For a cyclic operad $\POp$, that we may consider as a modular operad concentrated in genus zero, the cyclic bar construction $\Bar\POp$ is the genus zero part of the Feynman transform 
\[
    \Bar\POp ((r)) = \Feyn(\POp)((0,r)).    
\]
In particular, $\Bar\POp$ is a 1-shifted cyclic cooperad.
More precisely, it is the cofree 1-shifted coaugmented cyclic cooperad cogenerated by the augmentation ideal $\overline{\POp}$, equipped with a differential encoding the operad structure on $\POp$.
\[
    \Bar\POp = (\mathrm{Free}^c_\kk(\overline \POp), d)
\]

Dually, the cobar construction of a cyclic cooperad $\COp$
\[
    \Bar^c\COp = (\mathrm{Free}_\kk(\overline \COp), d)
\]
is the cofree 1-shifted augmented cyclic operad generated by $\overline \COp$, again with a differential encoding the cooperad structure on $\COp$.

The bar and cobar construction have natural 1-shifted variants.
For example, for $\COp$ a 1-shifted cyclic cooperad the cobar construction
\[
    \Bar^c\COp = (\mathrm{Free}(\overline \COp), d)
\]
is the free cyclic operad generated by $\overline\COp$ with a suitable differential.

\subsection{$\infty$-operads}
The notion of modular and cyclic $\infty$-operads has been studied in \cite{WardMassey}, generalizing the non-cyclic case \cite{vanderLaan}.
We just recall that a 1-shifted cyclic $\infty$-operad structure on $\POp$ is equivalent data to a degree 1 coderivation
\[
D : \mathrm{Free}^c(\overline \POp) \to \mathrm{Free}^c(\overline \POp)
\]
on the cofree coaugmented cyclic cooperad cogenerated by $\overline \POp$ satisfying $D^2=0$.
By cofreeness, any such coderivation $D$ is uniquely determined by its composition $\pi\circ D$ with the projection to cogenerators $\pi: \mathrm{Free}^c_\kk(\overline \POp) \to \POp$.
Hence the cyclic $\infty$-operad structure amounts to specifying composition operations 
\[
\mu_T : \otimes_T \overline \POp \to \overline \POp    
\]
for every non-rooted tree $T$.
We will use only cyclic $\infty$-operad structures with at most ternary operations, and the relevant trees are:
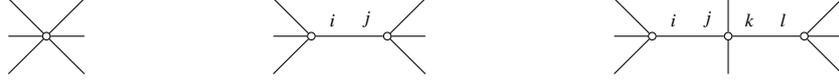
\begin{align*}
    &
\begin{tikzpicture}
    \node [ext] (v) at (0,0) {};
    \draw (v) edge +(-.5,.5) edge +(-.5,-.5) edge +(-.5,0)
    edge +(.5,.5) edge +(.5,-.5) edge +(.5,0);
\end{tikzpicture}
& &
\begin{tikzpicture}
    \node [ext] (v) at (0,0) {};
    \node [ext] (w) at (1,0) {};
    \draw (v) edge 
    node[above, near end] {$\scriptstyle j$}  
    node[above, near start] {$\scriptstyle i$} (w) 
    edge +(-.5,.5) edge +(-.5,-.5) edge +(-.5,0)
    (w) edge +(.5,.5) edge +(.5,-.5) edge +(.5,0);
\end{tikzpicture}
& &
\begin{tikzpicture}
    \node [ext] (v) at (0,0) {};
    \node [ext] (w) at (1,0) {};
    \node [ext] (u) at (2,0) {};
    \draw (v) edge 
    node[above, near end] {$\scriptstyle j$}  
    node[above, near start] {$\scriptstyle i$} (w) 
    edge +(-.5,.5) edge +(-.5,-.5) edge +(-.5,0)
    (w) edge node[above, near end] {$\scriptstyle l$}  
    node[above, near start] {$\scriptstyle k$} (u) 
    edge +(0,.5) edge +(0,-.5)
    (u) edge +(.5,.5) edge +(.5,-.5) edge +(.5,0);
\end{tikzpicture}
\end{align*}
The first corresponds to the differential $d$ on $\POp$, the second to the cyclic operadic composition $\mu_{i,j}$ and the third represents a ternary composition operation
\[
\mu_{i,j,k,l}: \POp((A\sqcup\{i\} )) \otimes \POp((B\sqcup\{j,k\})) \otimes Q\POp((C\sqcup\{l\}))
\to \POp((A\sqcup B\sqcup C))[1].
\]

For a 1-shifted cyclic $\infty$-operad $\POp$ we define its bar construction as 
\[
\Bar \POp = (\mathrm{Free}^c(\overline \POp) , D)
\]

An $\infty$-morphism between 1-shifted cyclic $\infty$-operads $\POp_1$ and $\POp_2$ is a morphism of coaugmented cyclic cooperads 
\[
F:  \Bar \POp_1 \to \Bar\POp_2.
\]
Again by cofreeness, the morphism $F$ is completely determined by its composition $\pi\circ F$ with the projection to cogenerators $\overline \POp_2$.
We say that the $\infty$-morphism $F$ is a quasi-isomorphism if the restriction 
\[
    \pi\circ F\mid_{\overline \POp_1} : \overline\POp_1
    \to 
    \overline \POp_2
\]
is a quasi-isomorphism of dg symmetric sequences.


For $\POp$ a 1-shifted cyclic $\infty$-operad there is always a canonical $\infty$-quasi-isomorphism
\[
\POp \to \Bar^c\Bar \POp    
\]
into the bar-cobar construction of $\POp$. It is given by the unit of the bar-cobar adjunction
\[
\Bar\POp \to \Bar\Bar^c\Bar \POp.   
\]
Since $\infty$-quasi-isomorphism can be inverted up to homotopy, there is also an $\infty$-quasi-isomorphism 
\[
\Bar^c\Bar \POp \to \POp,
\]
canonical up to homotopy. The object $\Bar^c\Bar \POp$ is an honest 1-shifted cyclic operad and is sometimes called the rectification of $\POp$. In particular any $\infty$-operad can be replaced by an $\infty$-quasi-isomorphic honest operad.

\subsection{The BV operad}
\label{sec:BV}
The Batalin-Vilkovisky operad $\BV$ is the homology operad of the framed little disks operad. 
It governs $\BV$-algebras, that is, graded vector spaces $V$ together with a graded commutative product $-\wedge-$, a bracket $[-,-]$ of degree $-1$ and a $\BV$ operator $\Delta:V\to V$ of degree $-1$ satisfying the following conditions:
\begin{itemize}
\item The product $\wedge$ makes $V$ into a graded commutative algebra.
\item The bracket $[-,-]$ is a Lie bracket of degree $-1$ on $V$.
\item There are the compatibility relations 
\begin{align*}
   \Delta^2&=0 & \Delta(x\wedge y)-(\Delta x)\wedge y - (-1)^{|x|}x\wedge (\Delta y)&=[x,y]\\
    [x,y\wedge z] &= [x,y]\wedge z + (-1)^{(|x|+1)|y|}y\wedge[x,z].
\end{align*}
\end{itemize}

For more details on the operad $\BV$ see \cite[Chapter 13.7]{LodayVallette}.

We declare a non-negative weight grading on the operad $\BV$ by defining the weight of a hmogeneous element to be minus twice the cohomological degree.

We will furthermore need an explicit basis of the part of weight $\leq 2$, $\gr_{\leq 2}\BV((n))$, introduced in \cite[section 6.2]{BrueckBorinskyWillwacher}, that we briefly recall.
The weight zero sub-operad $\gr_{0}\BV\cong \Com$ may be identified with the commutative operad $\Com$. In particular, $\gr_{0}\BV((n))=\Q c_n$ is one-dimensional, and we denote the generator by $c_n$, the $n-2$-fold commutative product.
The action of the symmetric group $S_n$ on $\gr_{0}\BV$ is trivial.

The space $\gr_2\BV((n))$ is $n\choose 2$-dimensional, and has a basis $(E_{ij})$ by symbols $E_{ij}=E_{ji}$ for $1\leq i\neq j\leq n$.
The action of the symmetric group $S_n$ on this basis is by permuting indices. The operadic compositions are as follows:

\begin{align*}
c_m \circ_{i,j} c_n &= c_{m+n} \\
E_{ij}\circ_{j,b} c_{A\sqcup b} &= \sum_{k\in A} E_{ik}.
\end{align*}

For example, the $\BV$ operator $\Delta\in \gr_2\BV((2))$ is expressed in this basis as $\Delta=E_{12}$

Dually, we have the dual basis $c_n^*, E_{ij}^*$ of the dual cooperad $\gr_{\leq 2}\BV((n))$.

\subsection{The gravity operad}
\label{sec:gravity def}
Consider the subspace 
\[
I_n = \mathrm{span} \{ x\circ_{i,1} \Delta | i=1,\dots,n;x\in \BV((n))\} \subset \BV((n))
\]
and the collection of degree shifted quotient spaces 
\[
\tGrav((n)) :=  \left(\BV((n)) / I_n\right) [6-2n].   
\]
This collection is a 1-shifted cyclic operad with the composition 
\[
[x]\circ_{i,j} [y] := [x\circ_{i,1}\Delta\circ_{2,j}y],    
\]
the gravity operad introduced in \cite{Getzler0}.
It is easy to check that 
\[
\tGrav((n)) \cong H_c^\bullet(\M_{0,n})
\]
can be identified with the compactly supported cohomology of the moduli space of genus zero curves.


\section{The $Q$-construction of a cyclic operad}
\label{sec:Qdual}

For this section, let $\POp$ be any cyclic operad such that $\POp((2))\cong H_\bullet(S^1) = \BV((2))$.
Of course, our main example is $\POp=\BV$, but working more generally allows for using different models for $\BV$ than $\BV$ itself.
We denote the standard basis of $H_\bullet(S^1)$ by $\{1,\Delta\}$ as usual, with $\Delta$ of degree $-1$.

We then define the dg symmetric sequence 
\[
Q\POp((r)) := 
\begin{cases}
    (\POp((r))[6-2r] \otimes \Q[v_1,\dots,v_r], d_{\POp} + d_v) & \text{if $r\geq 3$} \\
    \Q[u][1] & \text{if $r=2$}  
\end{cases},
\]
where $v_j$ are formal variables of cohomological degree $-2$, $u$ is a formal variable of degree $+2$, and the differential acts such that 
\[
d_v (x\otimes p(v_1,\dots, v_r))
=
\sum_{j=1}^r
(-1)^{|x|} (x\otimes_j \Delta) \otimes \frac{\partial}{\partial v_j} p(v_1,\dots, v_r).
\]
The $S_r$-action on $\POp((r))$ for $r\geq 2$ extends naturally to an $S_r$-action on $Q\POp((r))$ -- in particular, permutations permute the indices of the variables $v_j$.
The $S_2$ action on $Q\POp((2))$ is defined such that the transposition $\tau \in S_2$ acts as $\tau 1=1$ and 
\begin{equation}\label{equ:tau u}
\tau u^n :=  
(-1)^{n+1} u^n
\end{equation}
for $n\geq 1$.

We define the following binary operadic composition of degree $+1$ on $Q\POp$. First suppose that $A$, $B$ are finite sets with $|A|,|B|\geq 3$. Then we set:
\begin{gather*}
\mu_{a,b} : Q\POp((A)) \otimes Q\POp((B))
\to Q\POp((A\sqcup B\setminus\{a,b\})) \\
\mu_{a,b}(x\otimes p , y\otimes q)
:=
(-1)^{|x|}(x \circ_{a,1} \Delta \circ_{2,b} y) \otimes p q\mid_{v_a=v_b=0}.
\end{gather*}
Here $\circ_{i,j}$ is the operadic composition in $\POp$, and the notation $p q\mid_{v_a=v_b=0}$ means that we set the variables $v_a$ and $v_b$ to zero in the polynomial $pq$.

Next, if one of the factors is of arity $|A|\geq 3$ and the other of arity $2$ we set 
\begin{gather*}
\mu_{a,1} : Q\POp((A)) \otimes Q\POp((2))
\to Q\POp((A))\\
\mu_{a,1}( x\otimes p, u^n )
= (-1)^{|x|} x\otimes \partial_{v_a}^n p.
\end{gather*}
Note that \eqref{equ:tau u} and compatibility with the symmetric group action then implies that 
\[
\mu_{2,a}(u^n, x\otimes p) = (-1)^{n+1}x\otimes \partial_{v_a}^np.
\]

Finally, if both factors are of arity $2$ we define the composition as minus the product
\begin{gather*}
    \mu_{2,1} : Q\POp((2)) \otimes Q\POp((2))
    \to Q\POp((2))\\
    \mu_{2,1}( p(u), q(u) )
    = - p(u)q(u).
\end{gather*}

Next, we define the ternary compositions
\begin{gather*}
    \mu_{a,b_1,b_2,c} : Q\POp((A)) \otimes Q\POp((B))\otimes Q\POp((C))
    \to Q\POp((A\sqcup B\sqcup C\setminus\{a,b_1,b_2,c\})).
\end{gather*}
Concretely, we set 
\[
    \mu_{a,b_1,b_2,c}\equiv 0  
\]
if $|B|\geq 3$ or $|A|=2$ or $|C|=2$.
In the remaining case $|A|,|C|\geq 3$, $|B|=2$ we define the ternary composition as 

\begin{gather*}
    \mu_{a,1,2,c} : Q\POp((A)) \otimes Q\POp((2)) \otimes Q\POp((C))
    \to Q\POp((A\sqcup C\setminus\{a,c\})) \\
    \mu_{a,b}(x\otimes p , u^n, y\otimes r)
    :=
     (x \circ_{a,c} y) \otimes \sum_{j=0}^{n-1} 
     (-1)^{|x|+n-1-j}
     \partial_{v_a}^jp\mid_{v_a=0} \partial_{v_c}^{n-1-j} r \mid_{v_c=0}.
\end{gather*}

\begin{prop}\label{prop:Q well defined}
    For $\POp$ a cyclic dg operad such that $\POp((2))=H_\bullet(S^1)$ the dg symmetric sequence $Q\POp$ with the binary and ternary compositions $\mu_{-,-}$, $\mu_{-,-,-,-}$ defined above is a well-defined cyclic 1-shifted homotopy operad.
\end{prop}
\begin{proof}
    We leave it to the reader to check that the differential and composition morphism are compatible with the symmetric group actions.

    We only check that the $\infty$-operadic relations are satisfied.
    Since our $\infty$-operad has only composition morphisms up to arity 3, the only non-trivial relations to be checked are of arities $\leq 5$.
    We need to check them in turn.
    \begin{itemize}
    \item $\infty$-relations of arity 1: This relation states that $d^2=0$, but this is obviously true since $\Delta^2=0$ and composition is compatible with the differential $d_{\POp}$.
    \item Arity 2: This relation states that 
    \[
    -d\mu_{a,b}(X,Y) = \mu_{a,b}(dX, Y) + (-1)^{|X|}\mu_{a,b}(X, dY),    
    \] 
    or in other words that the differential anti-commutes with the binary operadic composition.\footnote{Mind that the binary composition has degree $+1$ for 1-shifted operads, hence the sign is different than one would have for ordinary operads.}
    But one sees that the only terms that can potentially be non-trivial in that anti-commutator have a $\cP$-factor of $x\circ \Delta\circ \Delta \circ y$, and hence vanish since $\Delta^2=0$.
    \item Arity 3:
    We have to check that
    \begin{multline}\label{equ:Qwdtbs3}
       - (-1)^{|X|}
      \mu_{a,b_1}(X, \mu_{b_2, c}(Y, Z))
      -
      \mu_{b_2, c}(\mu_{a,b_1}(X, Y), Z)
      \\=
      d\mu_{a,b_1,b_2,c}(X,Y,Z)
      +
      \mu_{a,b_1,b_2,c}(dX,Y,Z)
      + (-1)^{|X|}
      \mu_{a,b_1,b_2,c}(X,dY,Z)
      + (-1)^{|X|+|Y|}
      \mu_{a,b_1,b_2,c}(X,Y,dZ) .
    \end{multline}
    We have to distinguish several cases according to which of the three arguments $X,Y,Z$ have valence $=2$ or $\geq 3$.
    If either all have valence $\geq 3$ or at most one has valence $\geq 3$, then the right-hand side of \eqref{equ:Qwdtbs3} vanishes. and it is easy to see that the left-hand side vanishes as well.
    So we focus on the case that two of the arguments have arities $\geq 3$ and one has arity $=2$. If either $X$ or $Z$ have arity $=2$, then again both sides are easily seen to vanish.
    
    This leaves us with the only non-trivial case that $Y=u^{n+1}\in\Q\cP((2))$, while $X=x\otimes p$, $Z=z\otimes q$ have higher arity.
    The right-hand side of \eqref{equ:Qwdtbs3} then reads
    \[
        (x\circ_{a,1} \Delta \circ_{2,c} z) \otimes
        \left(
            \sum_{j=0}^{n} (-1)^{n-j}\partial_{v_a}^j\partial_{v_c}^{n-j+1}(\partial_{v_a}+\partial_{v_c})  pq\mid_{v_a=v_c=0} 
        \right)
        =
        (x\circ_{a,0} \Delta \circ_{1,c} z) \otimes
        \left(
             (-\partial_{v_a}^{n+1}-(-1)^n\partial_{v_c}^{n+1})  pq\mid_{v_a=v_c=0} 
        \right)
    \]
    by telescopic cancellation.
    But this is the same as the left-hand side of \eqref{equ:Qwdtbs3}. 

    \item arity 4:
    There are a priori two types of arity 4 relations corresponding to the two trees with four vertices 
    
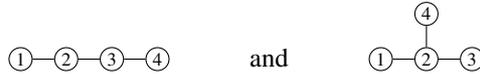
\[
    \begin{tikzpicture}
        \node[ext] (v1) at (0,0) {$\scriptstyle 1$};
        \node[ext] (v2) at (.6,0) {$\scriptstyle 2$};
        \node[ext] (v3) at (1.2,0) {$\scriptstyle 3$};
        \node[ext] (v4) at (1.8,0) {$\scriptstyle 4$};
        \draw (v2) edge (v3) edge (v1) (v3) edge (v4);
    \end{tikzpicture}
    \quad
    \quad
    \quad
    \text{and}
    \quad
    \quad
    \quad
    \begin{tikzpicture}
        \node[ext] (v1) at (0,0) {$\scriptstyle 1$};
        \node[ext] (v2) at (.6,0) {$\scriptstyle 2$};
        \node[ext] (v3) at (1.2,0) {$\scriptstyle 3$};
        \node[ext] (v4) at (.6,0.6) {$\scriptstyle 4$};
        \draw (v2) edge (v3) edge (v1) edge (v4);
    \end{tikzpicture}
    \]
    However, by arity reasons $Q\cP$ has no compositions that would give a nontrivial contribution for the second type of tree, so that we only need to consider the first "linear" tree.
    The corresponding relation reads
    \begin{multline}\label{equ:Qwdtbs4}
        (-1)^{|X|}\mu_{a,b_1}(X, \mu_{b_2, c_1,c_2,d}(Y, Z, W))
        +
        \mu_{b_2, c_1, c_2, d}(\mu_{a,b_1}(X, Y), Z, W)
        + (-1)^{|X|}
        \mu_{a,b_1, c_2, d}(X, \mu_{b_2, c_1}(Y, Z), W)
        \\+ (-1)^{|X|+|Y|}
        \mu_{a,b_1, b_2, c_1}(X, Y, \mu_{c_2, d}(Z, W))
        +
        \mu_{c_2,d}(\mu_{a, b_1, b_2, c_1}(X, Y, Z), W)
        =
        0.
      \end{multline}
      Again we need to consider in principle all possible $16$ combinations of $X,Y,Z,W$ being either of arity $2$ or of arity $\geq 3$. However, using symmetry, and since the ternary composition vanishes unless its middle argument is binary and the others are not, we can restrict to the following nontrivial sub-cases:
      \begin{itemize}
        \item $Y=u^m$ is binary and $X,Z,W$ are at least ternary: Here only the last two terms in \eqref{equ:Qwdtbs4} do not vanish and are identical (up to sign) by associativity of the operadic composition in $\cP$.
        \item $Y=u^m$ and $W=u^n$ are binary and $X,Z$ are at least ternary:
        Again the only non-vanishing terms are the last two, but they are again the same up to sign, since the ternary operation is defined independent of $v_{c_2}$ in the polynomial part of $Z$.
        \item $Y=u^{m+1}$ and $Z=u^{n+1}$ are binary and $X=x\otimes p,W=w\otimes q$ are at least ternary:
        This is the only really non-trivial case.
        Here the second, third and fourth terms of \eqref{equ:Qwdtbs4} non-zero and the equation reads:
        \begin{align*}
            &-(x\circ_{a, d}w) \otimes 
            \left(
                \sum_{j=0}^n 
                (-1)^{n-j}
                \partial_{v_a}^{m+1+j}
                \partial_{v_d}^{n-j}
                -
                \sum_{j=0}^{m+n+1}
                (-1)^{m+n+1-j} 
                \partial_{v_a}^{j}\partial_{v_d}^{m+n+1-j}
                +
                \sum_{j=0}^{m} 
                (-1)^{m+n+1-j}
                \partial_{v_a}^{j}\partial_{v_d}^{m+n+1-j}
            \right)
            pq\mid_{v_a=v_d=0}
            \\
            =&
            -(x\circ_{a, d}w) \otimes 
            \left(
                \sum_{j=m+1}^{n+m+1} 
                (-1)^{m+n+1-j}
                \partial_{v_a}^{j}
                \partial_{v_d}^{m+n+1-j}
                -
                \sum_{j=0}^{m+n+1}
                (-1)^{m+n+1-j} 
                \partial_{v_a}^{j}\partial_{v_d}^{m+n+1-j}
                +
                \sum_{j=0}^{m} 
                (-1)^{m+n+1-j}
                \partial_{v_a}^{j}\partial_{v_d}^{m+n+1-j}
            \right)
            pq\mid_{v_a=v_d=0}
            \\&=0
        \end{align*}
      \end{itemize}
      \item arity 5:
      Of the 3 possible trees with 5 vertices only the linear tree 
      \[
        \begin{tikzpicture}
            \node[ext] (v1) at (0,0) {$\scriptstyle 1$};
            \node[ext] (v2) at (.6,0) {$\scriptstyle 2$};
            \node[ext] (v3) at (1.2,0) {$\scriptstyle 3$};
            \node[ext] (v4) at (1.8,0) {$\scriptstyle 4$};
            \node[ext] (v5) at (2.4,0) {$\scriptstyle 5$};
            \draw (v2) edge (v3) edge (v1) (v4) edge (v3) edge (v5);
        \end{tikzpicture}
      \]
      contributes potentially non-trivial relations to be verified.
      The relation reads:
      \begin{multline}\label{equ:Qwdtbs5}
        (-1)^{|X|+|Y|}\mu_{a,b_1,b_2,c_1}(X, Y, \mu_{c_2, d_1,d_2,e}(Z, V, W))
        +(-1)^{|X|}
        \mu_{a,b_1,d_2,e}(X, \mu_{b_2, c_1,c_2,d_1}(Y, Z, V), W)
        \\+
        \mu_{c_2,d_1,d_2,e}(\mu_{a, b_1,b_2,c_1}(X, Y, Z), V, W)
        =
        0.
      \end{multline}
      The middle term is always zero, and the other terms can only contribute non-trivially if $Y=u^{m+1}$ and $V=u^{n+1}$ are binary and $X=x\otimes p$, $Z=z\otimes q$ and $W=w\otimes r$ are at least ternary.
      But then both the first and third term are equal (up to opposite sign) to
      \[
        (x\otimes_{a,c_1} z \otimes_{c_2,e} w)
        \otimes 
        \left(
            \sum_{i=0}^m\sum_{j=0}^n
            (-1)^{m+n-i-j}
            \partial_{v_a}^{i}\partial_{v_{c_1}}^{m-i}
            \partial_{v_{c_2}}^{j}\partial_{v_{e}}^{n-j}
            pqr\mid_{v_a=v_{c_1}=v_{c_2}=v_e=0}
        \right).
      \] 
    \end{itemize}

\end{proof}

\subsection{The weight grading on $Q\POp$}
\label{sec:Q weight}
We next suppose that $\POp$ comes equipped with an additional grading that we call the weight grading. We suppose that $\Delta\in \POp((2))$ is a homogeneous element of weight $2$.
We may then extend the weight grading to $Q\POp$ as follows.
We declare the generator $u\in Q\POp((2))$ to be of weight $W(u)=+2$.
For $n\geq 3$ let 
\[
x\otimes p\in Q\POp((n))=\POp((n))[6-2n]\otimes \Q[v_1,\dots,v_n]
\]
be an element such that $x$ is homogeneous of weight $W(x)$ in $\POp((n))$ and $p$ is a homogeneous polynomial of degree $\mathrm{deg}(p)$.
Then we define the weight of $x\otimes p$ in $Q\POp((n))$ as 
\[
W(x\otimes p) := 2n-6 - W(x) -2 \mathrm{deg}(p).   
\]
In other words, we reverse the weights of elements of $\POp$, and declare the symbols $v_j$ to carry weight $-2$ each.
The following is evident by inspecting the formulas.
\begin{lemma}
    The above weight grading on $Q\POp$ is compatible with the 1-shifted cyclic $\infty$-operad structure on $Q\POp$, in the sense that the structure operations (differential and compositions) all preserve the weight grading.
\end{lemma}

\subsection{Remark: $Q\BV$ and the gravity operad}
\label{sec:gravity}
Our main example will be the case $\POp=\BV$.
Let us first consider the non-unital 1-shifted cyclic $\infty$-suboperad $Q\BV_{\geq 3}\subset Q\BV$ of $Q\BV$ defined such that  
\[
    Q\BV_{\geq 3}((r))
=
\begin{cases}
Q\BV((r)) & \text{for $r\geq 3$} \\
0 &\text{otherwise}
\end{cases}.
\]
Note that the ternary composition vanishes on $Q\BV_{\geq 3}$, so that $Q\BV_{\geq 3}$ is an honest (non-unital) 1-shifted cyclic operad.
Furthermore, one has the natural projection map 
\[
    Q\BV_{\geq 3} \to \tGrav
\]
to the gravity operad.
This is in fact a quasi-isomorphism and compatible with the operadic composition.

\section{The proof of Theorem \ref{thm:main1}}
\label{sec:proof1}
To show Theorem \ref{thm:main1} we first construct an explicit weight preserving quasi-isomorphism of coaugmented cyclic cooperads
\[
F : \BV^* \to \Bar Q\BV.
\]

To construct the morphism $F$ it is sufficient to define the composition with the projection to cogenerators 
\[
f: \overline{\BV}^* \xrightarrow{F} \Bar Q\BV \xrightarrow{\pi} \overline { Q\BV}.    
\]
On the arity 2 part this morphism is defined by 
\begin{gather*}
    f: \overline\BV^*((2)) \cong  \Q \Delta^* \to  \overline{Q\BV}((2)) \cong u\Q[u] [1] \\
    f(\Delta^*) = u.
\end{gather*}
On the higher arity part the morphism $f$ factors through $\Com^*$.
\[
f : \BV^*((n)) \to \Com^*((n)) \xrightarrow{g} Q\BV((n)) \cong \BV((n))[6-2n]\otimes \Q[v_1,\dots,v_n].
\]
Here the second arrow $g$ is uniquely determined by specifying the image $g(c_n^*)$ of the cogenerator $c_n^*\in \Com^*((n))$. We define $f$ by setting 
\[
g(c_n^*) := c_n \otimes p_n,    
\]
with $p_n$ the symmetric polynomial 
\[
p_n 
=
(-1)^{n-3}
\sum_{
    \underline i=(i_0,i_1,\dots) \in \mathbb Z^n_{\geq 0} 
    \atop 
    {
    |\underline i|=n
    \atop 
    |\lambda(\underline i)|=n-3
    }
}
\sum_{\sigma \in S_n}
\frac{(n-3)!}{\underline i!}
\frac {\sigma \cdot v^{\lambda(\underline i)}}{\lambda(\underline i)!^2}
\in \Q[v_1,\dots,v_n]^{S_n},
\]
using the notation of section \ref{sec:multiind}.
For example:
\begin{align*}
p_3 &= 1 & p_4 &= -(v_1+v_2+v_3+v_4).
\end{align*}
It will be important to rewrite the coefficients in this definition in a more geometric way, using the notation of section \ref{sec:psi ints} as
\[
p_n 
=
(-1)^{n-3}
\sum_{
    \underline i=(i_0,i_1,\dots) \in \mathbb Z^n_{\geq 0} 
    \atop 
    {
    |\underline i|=n
    \atop 
    |\lambda(\underline i)|=n-3
    }
}
\sum_{\sigma \in S_n}
\frac{1}{\underline i!}
\frac {\sigma \cdot v^{\lambda(\underline i)}}{\lambda(\underline i)!}
\langle \tau^{\underline i}  \rangle_0
\]

\begin{prop}\label{prop:F morphism}
    The morphism $F: \BV^* \to DQ\BV$ defined above is a morphism of dg cyclic cooperads.
\end{prop}

The proof will rely on the following Lemma.
\begin{lemma}\label{lem:poly top recur}
Let $A=\{a_1,\dots,a_n\}$ be a finite set and write $p_A:=p_{n}(v_{a_1},\dots, v_{a_n})$.
Then we have for any $i\neq j\in A$
\begin{equation}\label{equ:lempoly}
    (\partial_{v_i}+\partial_{v_j}) p_A
    +
    \sum_{A_1\sqcup A_2 = A\atop
        i\in A_1, j\in A_2}
        p_{A_1\sqcup\alpha}p_{A_2\sqcup \beta}\mid_{v_\alpha=v_\beta=0}
    \, \,=0
\end{equation}
\end{lemma}
\begin{proof}
Extracting the coefficient of $v_i^av_j^bv^{\lambda(\underline j)}$ from both terms of \eqref{equ:lempoly}, it suffices to show that
\[
    \frac{1}{\underline j!}
    \langle (\tau_{a+1}\tau_b+\tau_a\tau_{b+1})\tau^{\underline j}  \rangle_0
    =
    \sum_{\underline j'+\underline j'' = j}
    \frac{1}{\underline j'!}
    \frac{1}{\underline j''!}
    \langle \tau_{a}\tau_0\tau^{\underline j'}  \rangle_0
    \langle \tau_b\tau_0\tau^{\underline j''}  \rangle_0.
\]
But this equation is the topological recursion relation for intersection numbers in the form \eqref{equ:top recur 2}.
\end{proof}

\begin{proof}[Proof of Proposition \ref{prop:F morphism}]
By construction it is clear that $F$ is compatible with the cyclic cooperad structures. We just need to check that $F$ is compatible with the differential, that is 
\[
dF(x) =0    
\]
for all $x\in \BV^*$.
As usual, it suffices to check the composition with the projection to cogenerators 
\begin{equation}\label{equ:Ftbs1}
\pi(dF(x))=0.    
\end{equation}
Furthermore, this holds obviously for $x\in \BV^*((2))$, so we restrict our attention to $x$ of higher arity $\geq 3$. 

Then equation \eqref{equ:Ftbs1} is equivalent to showing that
\begin{equation}\label{equ:Ftbs2}
df(x)
+ \frac1{2}
\sum \mu_{\alpha,\beta}(f(x') , f(x''))
 +\frac1{2}
\sum \mu_{\alpha,\beta_1,\beta_2,\gamma}(f(x') , f(x''), f(x''')) =0.
\end{equation}
Here we used Sweedler notation for the reduced cooperadic cocompositions
\begin{gather*}
    \overline{\BV}^*((A)) \to \oplus_{A=A_1\sqcup A_2} \overline{\BV}^*((A_1 \sqcup \{\alpha\})) \otimes \overline{\BV}^*((A_2 \sqcup \{\beta\})) \\
x\mapsto  \sum x'\otimes x''
\end{gather*}
and  
\begin{gather*}
    \overline{\BV}^*((A)) \to \oplus_{A=A_1\sqcup A_2\sqcup A_3} \overline{\BV}^*((A_1 \sqcup \{\alpha\})) \otimes \overline{\BV}^*((A_2 \sqcup \{\beta_1,\beta_2\}))\otimes \overline{\BV}^*((A_3 \sqcup \{\gamma\})) \\
    x\mapsto  \sum x'\otimes x'' \otimes x''',
\end{gather*}
with $\alpha,\beta,\beta_1,\beta_2,\gamma$ some symbols not in the set $A$. The factors $\frac12$ in \eqref{equ:Ftbs2} correct for the fact that our cocompositions, as defined above, pick up every partition of the set $A$ twice, while it should appear only once.

Next, note that $f(y)=0$ on all $y\in \BV^*$ of degree $\geq 2$.
Hence we need to check \eqref{equ:Ftbs2} only on $x$ of degrees $\leq 3$ by degree reasons.
Furthermore, for $|y|=1$, $f(y)$ can only be nonzero if $y$ is of arity $2$.

But then for $x$ of degree 3 only the last term in \eqref{equ:Ftbs2} could be non-zero, and it does not contribute because the ternary composition is only nonzero if the two outer arguments are of arity $\geq 3$.

Similarly, suppose that $x$ in \eqref{equ:Ftbs2} is of degree 2. Then two of $f(x'), f(x''), f(x')$ need to be of arity two, whence the third term cannot contribute.
The second can also not contribute, because $f(x')$ and  $f(x'')$ both would need to be of arity two, and hence also $x$, but there is no element $x\in\BV((2))$ of degree 2 to start with.

So we are left with the nontrivial cases of $x$ being of degree 0 or 1.
Suppose first that $x\in \BV^*((A))$ is of degree zero, with $|A|=:r\geq 3$. I.e., without loss of generality $x=c_A^*$ is identified with the $\Com^*$-cogenerator inside $\BV^*$.
Then the third term of \eqref{equ:Ftbs2} vanishes by degree reasons and \eqref{equ:Ftbs2} becomes
\begin{equation}\label{equ:Ftbs3}
\begin{aligned}
0&=d (c_A \otimes p_A)
+\frac12
\sum_{A_1\sqcup A_2=A} \mu_{\alpha,\beta}(c_{A_1\sqcup\alpha}\otimes p_{A_1\sqcup \alpha}, c_{A_2\sqcup \beta}\otimes p_{A_2\sqcup \beta} )
\\
&=d (c_A \otimes p_A)+\frac12
\sum_{A_1\sqcup A_2=A}
c_{A_1\sqcup\alpha}\circ_{\alpha, 1}\Delta\circ_{2,\beta}c_{A_2\sqcup \beta}
\otimes p_{A_1\sqcup \alpha}p_{A_2\sqcup \beta}\mid_{x_\alpha=x_\beta=0}.
\end{aligned}
\end{equation}
Now, using the basis $(E_{ij})_{ij}$ of the degree $-1$ part of $\BV$ of section \ref{sec:BV} we have that 
\[
d (c_A \otimes p_A)
=
\sum_{a\neq b\in A} E_{ab} \otimes \partial_a p_A
=
\frac 12 \sum_{a\neq b\in A} E_{ab} \otimes (\partial_a+\partial_b) p_A
\]
and 
\[
    c_{A_1\sqcup\alpha}\circ_{\alpha, 1}\Delta\circ_{2,\beta}c_{A_2\sqcup \beta}
    =
    \sum_{a\in A_1 \atop b\in A_2}
    E_{ab}. 
\]
Hence, collecting the coefficient of $E_{ab}$ in \eqref{equ:Ftbs3}, we find that \eqref{equ:Ftbs3} holds if and only if for all $a\neq b\in A$ we have 
\begin{equation}\label{equ:Ftbs4}
(\partial_a+\partial_b) p_A
=
\sum_{A_1\sqcup A_2 = A\atop
    a\in A_1, b\in A_2}
    p_{A_1\sqcup\alpha}p_{A_2\sqcup \beta}\mid_{x_\alpha=x_\beta=0}.
\end{equation}
But this is precisely Lemma \ref{lem:poly top recur}.

Finally we consider $x\in \BV^*((A))$ of degree 1, i.e., $x=E_{ij}^*$ for some $i\neq j\in A$.
The first term of \eqref{equ:Ftbs2} then vanishes.
Note also that $f(y)=0$ for all $y$ of degree $\geq 1$ and arity $\geq 3$. Hence the second term of \eqref{equ:Ftbs2} can only contribute if $x'$ or $x''$ are of arity $2$ and degree $1$. Furthemore, there are only two cocompositions that produce such terms from $E_{ij}^*$, so that the second term of \eqref{equ:Ftbs2} becomes
\[
c_A  \otimes (\partial_i+\partial_j) p_A.
\]
Similarly, the third term of \eqref{equ:Ftbs2} can only contribute if $x''$ is of arity two and degree one, and the only possible cocompositions are such that $i$ is appearing in $x'$ and $j$ in $x'''$ or vice versa.
The third term of \eqref{equ:Ftbs2} is hence equal to 
\[
    \sum_{A_1\sqcup A_2 = A\atop
    a\in A_1, b\in A_2}
    p_{A_1\sqcup\alpha}p_{A_2\sqcup \beta}\mid_{x_\alpha=x_\beta=0}.
\]

But this means that \eqref{equ:Ftbs2} is equivalent to \eqref{equ:Ftbs4}, and we are hence done by Lemma \ref{lem:poly top recur}.

\end{proof}

\begin{proof}[Proof of Theorem \ref{thm:main1}]
    Applying the cobar construction to the map $F$ of Proposition \ref{prop:F morphism} yields a morphism of cyclic operads
    \[
    \Bar^c F : \Bar^c\BV^* \to \Bar^cQ\BV.     
    \]
The cyclic operadic $\infty$-morphism $\Bar^c\BV^*\to Q\BV$ of Theorem \ref{thm:main1} is the composition
\[
    \Bar^c\BV^* \xrightarrow{DF} \Bar^c \Bar Q\BV \xrightarrow{\sim} Q\BV
\]
of $\Bar^cF$ with the (up to homotopy) canonical cyclic operadic $\infty$-quasi-isomorphism $\Bar^c\Bar Q\BV \to Q\BV$.
We claim that the above composition is a quasi-isomorphism.
To this end, let us endow $H(\Bar^c\BV^*)$ with a 1-shifted cyclic $\infty$-operad structure by homotopy transfer from $\Bar^c\BV^*$. Consider the $\infty$-morphism 
\[
    G:H(\Bar^c\BV^*)\to Q\BV    
\]
that is the composition 
\[
    H(\Bar^c\BV^*)\xrightarrow{\sim} \Bar^c\BV^* \xrightarrow{\Bar^cF} \Bar^c\Bar Q\BV \xrightarrow{\sim} Q\BV,
\]
with the left-most $\infty$-quasi-isomorphism obtained as part of the homotopy transfer.
The morphism $G$ also preserves the weight gradings, provided suitable choices in the homotopy transfer.
We have to check that the linear part $G_1$ of $G$ induces an isomorphism on cohomology 
\[
[G_1] : H(\Bar^c\BV^*) \to H(Q\BV).
\]
Note that by \cite[Theorem 2.21 and Proposition 3.9]{DCV} we have that
\[
    H(\Bar^c\BV^*((n)))  
    =
    \begin{cases}
    \Q[u][1] & \text{for $n=2$} \\
    H_c^\bullet(\M_{0,n}) & \text{for $n\geq 3$}
    \end{cases}.
\]
Hence essentially by the definition of $Q\BV$ we have that $H(\Bar^c\BV^*((n))) \cong H(Q\BV)$.
We next check that $[G_1]$ is an isomorphism in arities $n=2,3$:
\begin{itemize}
\item $n=2$: $H(\Bar^c\BV^*((2)))\cong \Q[u][1]\cong H(Q\BV((2)))$ is a 1-shifted algebra generated by the single element $u$. 
But tracing the generator $u$ through the composition defining $G$, one sees that $u$ is mapped to the generator of $H(Q\BV((2)))$. (The intermediate elements in the bar/cobar constructions are always trivial trees with one vertex.) Hence $[G_1]$ is an isomorphism in arity $2$.
\item $n=3$: $H(\Bar^c\BV^*((3)))\cong \Q \cong H(Q\BV((3)))$
is a one-dimensional space concentrated in degree zero. The representing cocycle in $D\BV^*((3))$ is a tree with one vertex decorated by the cocommutative coproduct generator $c_3^*\in \BV^*((3))$.
On the other side, the representing cocycle in $Q\BV((3))$ is the element $c_3\otimes 1$, with $c_3\in \BV((3))$ the commutative product generator.
Tracing through the composition of morphisms above, we see that the generators are mapped onto each other, so that $[G_1]$ is an isomorphism in arity $3$.
\end{itemize}

By finite dimensionality, it suffices to show that $[G_1]$ is surjective in arities $\geq 3$. But $[G_1]$ is also a morphism of operads. The operadic composition on 
$H(Q\BV_{\geq 3})\simeq H_c^\bullet(\M_{0,-})\cong \tGrav$ agrees with the operadic composition of the gravity operad, see section \ref{sec:gravity}. But Getzler \cite{Getzler0} has shown that the gravity operad is generated by its top weight parts $\gr_{2n-6} H_c^{2n-6}(\M_{0,n})\cong \Q$.
Hence it is sufficient to check that the map $[G_1]$ is surjective onto the parts of top weight ${2n-6}$.

We do this by induction on $n$. The base case $n=3$ has been checked above.
For the step $n\to n+1$ first note that the top weight part of $Q\BV((n))$ is in fact 1-dimensional, generated by the element $c_n\otimes 1$.
As the induction hypothesis we may hence assume that there is an element $x_n\in H(\Bar^c\BV^*((n)))$ so that 
\[
 G_1(x_n) = c_n\otimes 1, 
\]
and we want to construct $x_{n+1}$ satisfying $G_1(x_{n+1}) = c_{n+1}\otimes 1$.
Let $\mu'$ be the $\infty$-operadic structure on $H(\Bar^c\BV^*)$. Then we define 
\[
x_{n+1} := \mu'_{1,1,2,1}(x_n,u,x_3)
\]
as the ternary composition. Then by the relations satisfied by an $\infty$-morphism of cyclic operads we have that 
\begin{multline*}
G_1(x_{n+1}) - 
\mu_{1,1,2,1}(G_1(x_n), G_1(u), G_1(x_3))
\\=
\pm G_2(\mu'_{1,1}(x_n,u), x_3)
\pm G_2(x_n, \mu'_{2,1}(u,x_3))
\pm \mu_{2,1}(G_2(x_n,u), x_3)
\pm \mu_{1,1}(x_n, G_2(u,x_3))
\pm d G_3(x_n, u, x_3),
\end{multline*}
using the slightly abusive notation $G_k$ for the $k$-linear part of the $\infty$-morphism $G$.
In any case, the right-hand side vanishes by weight and degree reasons: $\mu'_{1,1}(x_n,u)$ is an element of $H(\Bar^c\BV^*((n)))$ of weight $2n-6+2$, but the top weight is $2n-6$ in arity $n$, and hence $\mu'_{1,1}(x_n,u)=0$. The same reasoning applies to $\mu'_{2,1}(u,x_3)=0$. Similarly, $Q\BV((r))$ is concentrated in weights $\leq 2n-6$, and the top weight part is one-dimensional. Hence $G_2(x_n,u)$ and $G_2(u,x_3)$ are of top$+2$-weight and thus zero, and the differential on the top weight part vanishes, so that also $d G_3(x_n, u, x_3)=0$.
Thus we have shown that 
\[
    G_1(x_{n+1}) = \mu_{1,1,2,1}(G_1(x_n), G_1(u), G_1(x_3))
    =
    \mu_{1,1,2,1}(c_n\otimes 1, u, c_3\otimes 1)
    =
    c_{n+1}\otimes 1
\]
as desired. It follows that $G$ is a quasi-isomorphism, and hence Theorem \ref{thm:main1} is proven. 



\end{proof}

As a Corollary, we also obtain:
\begin{prop}\label{prop:F qiso}
    The map $F:\BV^*\to \Bar Q\BV$ of Proposition \ref{prop:F morphism} is a quasi-isomorphism.
\end{prop}

\begin{proof}

    Consider the commutative diagram of cyclic cooperads
    \[
        \begin{tikzcd}
            \Bar \Bar^c\BV^* \ar{r}{\Bar\Bar^cF}[below]{\sim} & \Bar\Bar^c\Bar Q\BV \\
     \BV^*  \ar{u}{\sim}\ar{r}{F} & \Bar Q\BV\ar{u}{\sim}
     \end{tikzcd}.
    \]
    The vertical arrows are the units of the bar-cobar adjunction and are known to be quasi-isomorphisms.
    We know from the previous proof of Theorem \ref{thm:main1} that the morphism $\Bar^c F$ is a quasi-isomorphism. But since the bar construction preserves quasi-isomrophisms, $\Bar\Bar^cF$ is also a quasi-isomorphism. Hence $F$ is also a quasi-isomorphism.
    
    \end{proof}

\section{Comparing Feynman transforms}
\label{sec:feyn}



\begin{thm}\label{thm:feyn and Q}
Let $\cP$ be a cyclic operad such that $\POp((2))=H_\bullet(S^1)$.
Suppose that $\cP$ is equipped with an arity-wise bounded weight grading, compatible with the operad structure, such that the degree $-1$ generator $\Delta\in \POp((2))$ has weight 2.
Then there are quasi-isomorphisms
\[
\Phi_{g,n}: \gr_W \Feyn(\POp)((g,n)) \to \gr_{6g-6+2n-W} \AFeyn(Q\POp)((g,n))[6g-6+2n]
\]
for all $W$ and $(g,n)$ such that $2g+n\geq 3$.
\end{thm}
We will show the Theorem in several steps.
First, we recall the explicit expression for the Feynman transforms appearing in the Theorem.
As a graded vector space, 
\[
    \Feyn(\POp)((g,n))
    =
    \bigoplus_\Gamma 
    \bigotimes_{v\in V\Gamma}
    \overline\POp((S_v))
    \otimes \bigotimes_{e\in E\Gamma}\Q[1]/\sim.
\]
Here the sum is over graphs $\Gamma$ of loop order $g$ with $n$ numbered legs. The first tensor product is over vertices $v$ of $\Gamma$ and $S_v$ is the star of $v$, i.e., the set of half-edges incident to $v$.
The second tensor product is over internal edges and produces only a degree shift and sign factors. Finally, the equivalence relation identifies isomorphic graphs.
A priori the graphs $\Gamma$ can have bivalent vertices.
However, we may consider strings of bivalent vertices as a decoration on an edge as indicated in the following picture.
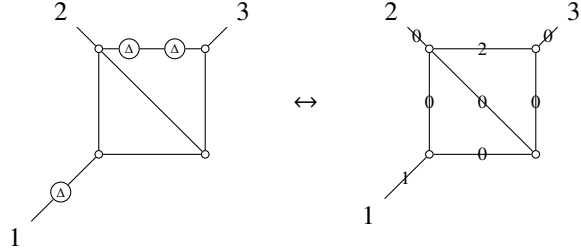
\[
        \begin{tikzpicture}
        \node[ext] (v1) at (-.7,-.7) {};
        \node[ext] (v2) at (+.7,-.7) {};
        \node[ext] (v3) at (-.7,+.7) {};
        \node[ext] (v4) at (+.7,+.7) {};
        \node[ext] (b1a) at (-.3,+.7) {$\scriptscriptstyle \Delta$};
        \node[ext] (b1b) at (.3,+.7) {$\scriptscriptstyle \Delta$};
        \node[ext] (b2) at (-1.2,-1.2) {$\scriptscriptstyle \Delta$};
        \node (n1) at (-1.8,-1.8) {$1$};
        \node (n2) at (-1.2,+1.2) {$2$};
        \node (n3) at (+1.2,+1.2) {$3$};
        \draw (v1) edge (v2) edge (v3)
        (b1a) edge (v3) edge (b1b)
        (b1b) edge (v4)
        (b2) edge (v1) edge (n1)
        (v4) edge (v2) edge (n3)
        (v2) edge (v3)
        (v3) edge (n2);
        \end{tikzpicture}
        \quad
        \leftrightarrow
        \quad
        \begin{tikzpicture}
          \node[ext] (v1) at (-.7,-.7) {};
          \node[ext] (v2) at (+.7,-.7) {};
          \node[ext] (v3) at (-.7,+.7) {};
          \node[ext] (v4) at (+.7,+.7) {};
          \node (n1) at (-1.5,-1.5) {$1$};
          \node (n2) at (-1.2,+1.2) {$2$};
          \node (n3) at (+1.2,+1.2) {$3$};
          \draw 
          (v1) edge node {$\scriptstyle 0$} (v2) 
          edge node {$\scriptstyle 0$} (v3) 
          edge node {$\scriptstyle 1$} (n1)
          (v4) edge node {$\scriptstyle 0$} (v2) 
          edge node {$\scriptstyle 0$} (n3) 
          edge node {$\scriptstyle 2$} (v3)
          (v2) edge node {$\scriptstyle 0$} (v3)
          (v3) edge node {$\scriptstyle 0$} (n2);
          \end{tikzpicture}
\]
We can thus rewrite for $(g,n)\neq (1,0)$
\begin{equation}\label{equ:FeynP2}
    \Feyn(\POp)((g,n))
    \cong
    \bigoplus_\Gamma '
    \bigotimes_{v\in V\Gamma}
    \POp((S_v))
    \otimes \bigotimes_{e\in E\Gamma}V_E
    \otimes \bigotimes_{h=1}^n V_L /\sim,
\end{equation}
where now the sum is only over graphs with at least trivalent vertices.
The graded vector space of decorations on the internal edges is 
\[
V_E :=  \bigoplus_{r\geq 0} (\Q[1]\otimes \overline \POp((2)))^{\otimes r} \otimes \Q[1]
\]
and the graded vector space associated to legs is
\[
V_L=\bigoplus_{r\geq 0} (\Q[1]\otimes \overline \POp((2)))^{\otimes r}.    
\]

The differential on $\Feyn(\POp)$ has the form 
\begin{equation}\label{equ:FeynP d}
d=d_{\POp}+d_c
\end{equation}
where $d_{\POp}$ is induced by the differential on $\POp$ and acts on the vertex decorations and $d_c$ acts by contracting an internal edge.
We may split the part $d_c$ further,
\[
d_c=d_{c,3} + d_{c,mix},
\]
with $d_{c,3}$ contracting an edge between two $\geq 3$-valent vertices, and $d_{c,mix}$ contracting an edge between a bivalent and a trivalent vertex.
Note that contracting an edge between two bivalent vertices yields zero because $\Delta^2=0$.

Next we turn to $\Feyn_\kk(Q\POp)((g,n))$. Proceeding similarly as before we rewrite
\begin{align}
    \nonumber
    \AFeyn_\kk(Q\POp)((g,n))
    &=
    \bigoplus_\Gamma 
    \bigotimes_{v\in V\Gamma}
    \overline{Q\POp}((S_v))/\sim
    \\ \nonumber
    &=
    \bigoplus_\Gamma '
    \bigotimes_{v\in V\Gamma}
    \left(\POp((S_v))[6-2|S_v|]\otimes \Q[v]^{\otimes S_v} \right)
    \otimes \bigotimes_{e\in E\Gamma} 
    \left(\bigoplus_{r\geq 0}
    u\Q[u][1]^{\otimes r}
    \right) 
    /\sim
    \\ \nonumber
    &=
    \bigoplus_\Gamma '
    \bigotimes_{v\in V\Gamma}
    \left(\POp((S_v))[6-2|S_v|]\right)
    \otimes \bigotimes_{e\in E\Gamma} 
    \underbrace{\left(
        \bigoplus_{r\geq 0}
        \Q[v]\otimes u\Q[u][1]^{\otimes r} \otimes \Q[v]
    \right) }_{=:V_E'[2]}
    \otimes 
    \bigotimes_{j=1}^n 
    \underbrace{
        \Q[v] 
    }_{=:V_L'[2]}
    /\sim
    \\&\cong \label{equ:FeynQP3}
    \left(\bigoplus_\Gamma '
    \bigotimes_{v\in V\Gamma}
    \POp((S_v))
    \otimes \bigotimes_{e\in E\Gamma} 
    V_E'
    \otimes 
    \bigotimes_{j=1}^n 
    V_L'
    /\sim\right)[6-6g-2n]
    .
\end{align}
Here we performed the following simplifications.
The first line is the definition of the amputated Feynman transform.
The outer sum is over graphs that do not have legs attached to bivalent vertices.
Note that we do not have a degree shift due to the edges owed to the fact that $Q\POp$ is a 1-shifted cyclic operad.
In the second line the outer sum runs only over $\geq 3$-valent graphs, and we instead consider strings of bivalent vertices as edge decorations as before.
In the third line we move the factors $\Q[v]$ that are part of the definition of $Q\POp$ into the corresponding edge or leg decorations. Finally, in the last line we have just taken out the degree shifts to reach an expression formally similar to \eqref{equ:FeynP2}.
The differential on $\Feyn_\kk(Q\POp)((g,n))$ now has the form 
\[
d = d_{Q\POp} + d_c + d_{cc}
\]
with $d_{Q\POp}=d_{\POp} + d_v$ induced from the differential on $Q\POp$,  $d_c$ acting by contracting edges and $d_{cc}$ acting by contracting strings of two edges and a bivalent vertex in the middle, applying the ternary composition on $Q\POp$.
\[
    d_{cc}\colon
\begin{tikzpicture}
\node[ext] (v) at (0,0) {$\scriptstyle X$};
\node[ext] (w) at (1.4,0) {$\scriptstyle Y$};
\node[ext] (m) at (0.7,0) {$\scriptstyle u^n$};
\draw (m) edge (v) edge (w) 
(w) edge +(.5,.5) edge +(.5,0) edge +(.5,-.5)
(v) edge +(-.5,.5) edge +(-.5,0) edge +(-.5,-.5);
\end{tikzpicture}
\mapsto 
\begin{tikzpicture}
    \node[ext] (w) at (1.4,0) {$\scriptstyle Z$};
    \draw 
    (w) edge +(.5,.5) edge +(.5,0) edge +(.5,-.5)
    edge +(-.5,.5) edge +(-.5,0) edge +(-.5,-.5);
 \end{tikzpicture}    
 \quad \text{ with }
 Z=\mu_{a,1,2,b}(X,u^n,Y)
\]
We may again rearrange the terms of the differential as
\begin{align}\label{equ:FeynQP d2}
d &= d_{\POp}
+d_{v}
+
d_{c,3}
+
d_{c,E} + d_{cc},
\end{align}
with $d_{c,3}$ the part of $d_c$ contracting an edge between two $\geq 3$-valent vertices and $d_{c,E}$ contracting edges incident to at least one bivalent vertex.

Now we construct $\Phi_{g,n}$ explicitly. This does not use the weight grading.
Comparing \eqref{equ:FeynP2} to \eqref{equ:FeynQP3} we define just $\Phi_{g,n}$ as the identity on the factors $\POp$, and by applying the following maps to the edge and vertex decorations:
\begin{gather}\label{equ:phiE def}
\phi_E : V_E \to V_E' \\
\Delta ^{\otimes k} \mapsto \sum_{j=0}^{k} 
\frac{(-1)^{k_e-j}}{j!(k_e-j)!} v^j \otimes u \otimes v^{k-j}
\in \Q[v] \otimes u\Q[u] \otimes \Q[v]\subset V_E'
\end{gather}
and 
\begin{gather*}
    \phi_L : V_L \to V_L' \\
    \Delta^{\otimes k} \mapsto 
    v^k
    \in \Q[v] \subset V_L'.
\end{gather*}

\begin{lemma}\label{lem:Phi def}
    For each $(g,n)$ such that $2g+n\geq 3$ the morphism $\Phi_{g,n}$ defined above is a well-defined morphism of complexes.
\end{lemma}
\begin{proof}
We first check that $\Phi_{g,n}$ is well-defined as a map of graded vector spaces, i.e., that the prescription above is compatible with taking the equivalence relations $\sim$. However, the action of isomorphisms on the edge decorations factors through the $S_2$-action by reversing the order of factors. 
So one just needs to check that the map $\phi_E$ is $S_2$-equivariant: The $k$-th summand in $V_E$ is one-dimensional. The generator $\Delta^{\otimes k}$, corresponding to a string of $k$ vertices and $k+1$ edges, each of degree $-1$, transforms with a sign $(-1)^k$ under the transposition. But the same is true for its image under $\phi_E$, see \eqref{equ:phiE def}. 

Next we check that $\Phi_{g,n}$ is compatible with the differentials. Comparing \eqref{equ:FeynP d} to \eqref{equ:FeynQP d2} we have to check that 
\[
    \Phi_{g,n} \circ (d_{\cP}+ d_{c,3} + d_{c,mix})
    =
    (d_{\cP}+ d_{v}+ d_{c,3}  + d_E + d_{cc})\circ \Phi_{g,n}.
\]
It is clear that $\Phi_{g,n} d_{\cP}=d_{\cP}\Phi_{g,n}$.
Furthermore, the image of $\phi_E$ is closed under $d_E$ so that $d_E \circ \Phi_{g,n}=0$.
Next, the image of $\phi_E$ always contains one "$u$", so that there is never an undecorated edge between $\geq 3$-valent vertices in the image of $\Phi_{g,n}$. Hence $d_{c,3} \circ \Phi_{g,n}=0$.
Looking at the definition of the ternary composition in $Q\POp$ and the image of $\phi_E$ we see that 
\[
    \Phi_{g,n} \circ d_{c,3} = d_{cc} \circ \Phi_{g,n}.
\] 
Finally, by direct inspection we have 
\[
    \Phi_{g,n} \circ d_{c,mix} = d_v \circ \Phi_{g,n},
\]
so that $\Phi_{g,n}$ indeed intertwines the differentials.
\end{proof}

\begin{lemma}\label{lem:Phi weight}
    If $\POp$ is equipped with a weight grading such that $\Delta$ has weight $+2$, then the morphisms $\Phi_{g,n}$ restrict to morphisms between the weight graded pieces 
    \[
        \Phi_{g,n}: \gr_W \Feyn(\POp)((g,n)) \to \gr_{6g-6+2n-W}\AFeyn(Q\POp)((g,n))[6g-6+2n]
    \]
    \end{lemma}
    \begin{proof}
        Let $\Gamma\in \Feyn(\POp)((g,n))$ be a graph with vertices $v$ of valence $k_v$ decorated by elements in $\gr_{W_v}\BV((k_v))$, $b$ bivalent vertices decorated by $\Delta$ and $e$ internal edges.
        Then $\Gamma$ is concentrated in weight
        \[
        W(\Gamma):=\sum_{v\in V\Gamma} W_v +2 b.    
        \] 
        Let $\Gamma'=\Phi_{g,n}(\Gamma)\in \AFeyn(Q\POp)((g,n))$ be the image under $\Phi_{g,n}$. Then, taking into account the definition of the weight grading on $Q\POp$ of section \ref{sec:Q weight}, $\Gamma'$ lives in weight 
        \[
            W(\Gamma'):=\sum_{v\in V\Gamma} 
        (2k_v -6-W_v) + 2e-2b
        =
        6g-6+2n - W(\Gamma).
        \]
    \end{proof}

    \begin{lemma}\label{lem:VEVL qiso}
        The morphisms $\phi_E:(V_E,0) \to (V_E',d_E)$ and $\phi_L:(V_L,0) \to (V_L',0)$ are quasi-isomorphisms.
    \end{lemma}
    \begin{proof}
        It is obvious that $\phi_L$ is an isomorphism, so we focus on $\phi_E$. 
        We compute the cohomology of $V_E'$ by filtering by the total number of $v$'s, and considering the associated (bounded below and exhaustive) spectral sequence. 
        The first page of the associated graded has the form 
        \[
            E^0=
        \Q[v] \otimes \left(
            \bigoplus_{k\geq 0} (u\Q[u][1])^{\otimes k)}
        , d_{bar} \right)  \otimes \Q[v]
        \]
        with the term in parentheses being isomorphic to the reduced bar construction of the associative algebra $\Q[u]$. The associative algebra $\Q[u]$ is Koszul, hence the cohomology of the bar construction is identified with the Koszul dual algebra $\Q[\epsilon] = \Q1\oplus \Q\epsilon$. Here $1$ corresponds to the generator of the $k=0$-summand and $\epsilon$ to $u$ in the $k=1$-summand of the bar construction.
        We find that 
        \[
        E^1 = \left(
        \Q[v] \otimes 
            (\Q \oplus \Q u)
          \otimes \Q[v], d_E \right).
        \] 
        
This complex has the explicit basis and differential 
\begin{align*}
    e_{n,a} := \frac{1}{a!(n-a)!} v^a \otimes 1\otimes v^{n-a} &\xrightarrow{d_E} 0 & &n\geq a\geq 0\\
    f_{n,a} := \frac{1}{a!(n+1-a)!} v^a \otimes u\otimes v^{n+1-a}
    &\xrightarrow{d_E}
    e_{n,a-1}+e_{n,a} & &n+1\geq a \geq 0,
\end{align*}
with the convention that $e_{n,-1}=e_{n,n+1}:=0$.
The basis elements for fixed $n$ span a subcomplex isomorphic to the mapping cone of the linear map
\begin{gather*}
\Q^{n+2} \to \Q^{n+1} \\
(x_0,\dots,x_{n+1}) \mapsto (x_0+x_1,x_1+x_2,\dots,x_{n+1}+x_n).
\end{gather*}
This map is surjective with one-dimensional kernel spanned by $(1,-1,\dots, (-1)^{n+1})$. But this is just the image of the map $\phi_E$ as in \eqref{equ:phiE def}, expressed in our basis. The spectral sequence abuts here by degree reasons, so that the Lemma follows.
    \end{proof}
    
    \begin{lemma}\label{lem:Phi qiso}
        Suppose that $\POp$ is equipped with an additional weight grading, that is arity-wise bounded, compatible with the operad structure and such that $\Delta$ is of weight $2$.
        Then for each $(g,n)$ such that $2g+n\geq 3$ the morphism $\Phi_{g,n}$ defined above is a quasi-isomorphism.
    \end{lemma}
    \begin{proof}
    We filter both the domain and target of $\Phi_{g,n}$ by the number of $\geq 3$-valent vertices in graphs, minus the total weight on the $\geq 3$-valent vertices.
    This filtration is bounded below and exhaustive, due to the boundedness assumption on the weight grading.
    Hence it suffices to check that $\Phi_{g,n}$ induces a quasi-isomorphism on the associated graded complexes:
    \begin{align*}
        \Phi_{g,n} : (\Feyn(\POp)((g,n)), d_{\POp}  )
    \to (\AFeyn_{\kk}(Q\POp)((g,n)), d_{\POp} + d_E ).
    \end{align*}
    The differentials here do not change the underlying $\geq 3$-valent graphs, but only act on the edge and vertex decorations.
    Hence the fact that the above associated graded morphism of $\Phi_{g,n}$ is a quasi-isomorphism follows from $\phi_E$ and $\phi_L$ being quasi-isomorphisms, but that has been shown in Lemma \ref{lem:VEVL qiso}.

    \end{proof}
    
    From the above Lemmas Theorem \ref{thm:feyn and Q} immediately follows.

    \begin{proof}[Proof of Theorem \ref{thm:main2}]
    We apply Theorem \ref{thm:feyn and Q} to the case of $\POp=\BV$, obtaining quasi-isomorphhisms (as long as  $2g+n\geq 3$)
    \[
        \Phi_{g,n}: \Feyn(\BV)((g,n)) \to \AFeyn_{\kk}(Q\BV)((g,n))[6g-6+2n].
    \]
    Furthermore, by Theorem \ref{thm:main1} we have a weight-preserving quasi-isomorphism of 1-shifted cyclic $\infty$-operads $D\BV^*\to Q\BV$. This induces morphisms between the amputated Feynman transforms 
    \[
        \gr_{6g-6+2n-W}\AFeyn_{\kk}(D\BV)((g,n))[6g-6+2n]
        \to 
        \gr_{6g-6+2n-W}\AFeyn_{\kk}(Q\BV)((g,n))[6g-6+2n].
    \]
    These morphisms are quasi-isomorphisms since the (amputated) Feynman transform preserves quasi-isomorphisms.
    Hence Theorem \ref{thm:main2} follows.
    \end{proof}

\bibliographystyle{amsalpha}

\end{document}